\documentclass[12pt]{amsart}

\usepackage{amssymb,bm,mathrsfs,mathtools,booktabs}

\usepackage{enumerate}
\usepackage{tikz}
\usepackage[centering,width=6in]{geometry}
\newtheorem{theorem}{Theorem}[section]
\newtheorem{lemma}[theorem]{Lemma}
\newtheorem{proposition}[theorem]{Proposition}
\newtheorem{corollary}[theorem]{Corollary}
\newtheorem*{theorem*}{Theorem}
\newtheorem{question}[theorem]{Question}
\newtheorem{conjecture}[theorem]{Conjecture}
\theoremstyle{definition}

\newcommand{\R}{\mathbb{R}}
\newcommand{\Q}{\mathbb{Q}}
\newcommand{\Z}{\mathbb{Z}}
\newcommand{\N}{\mathbb{N}}

\usepackage{csquotes}
\usepackage{hyperref}
\hypersetup{
    colorlinks,
    citecolor=blue,
    filecolor=black,
    linkcolor=red,
    urlcolor=black
}
\usepackage{enumitem}
\usepackage{thmtools} 
\usepackage{thm-restate}
\usepackage{dsfont}

\numberwithin{equation}{section}
\title[Fifty years of the Erd\H{o}s similarity conjecture]{Fifty years of the Erd\H{o}s similarity conjecture}

\author{Yeonwook Jung}
\author{Chun-Kit Lai}
\author{Yuveshen Mooroogen}

\begin{document}

\begin{abstract}
 Erd\H{o}s similarity conjecture was proposed by P. Erd\H{o}s in 1974. The conjecture remains open for exponentially decaying sequences as well as Cantor sets that have both Newhouse thickness and Hausdorff dimension zero.  In this article, written after 50 years of the conjecture being proposed, we review progress on some new variants of the original problem: namely, the bi-Lipschitz variant, the topological variant, and a variant ``in the large''. These problems were recently studied by the authors and their collaborators. Each of them offers new perspectives on the original conjecture. 
\end{abstract}

\maketitle

\section{Introduction}

It is a  general intuition that a large set or a large sample should contain many patterns. Many modern mathematical problems can fit into this framework by identifying appropriate notions of ``largeness" and  ``pattern". One of the most natural patterns we consider in this paper will be affine copies of a prescribed set. A non-trivial {\bf affine copy} of a set $A\subset \R$ is the image of an affine map $\lambda A+t$ where $\lambda\ne 0$ and $t\in \R$.  Let $\mathscr{S}$ be a collection of subsets in $\R$. We say that a set $A \subseteq \R^d$ is \textbf{universal in} $\mathscr{S}$ if every set $E \in \mathscr{S}$ contains a nontrivial affine copy of $A$.

In the continuous setting, a natural way to describe the ``largeness"  of sets is via their Lebesgue measure. We call a set $A \subset \R$ \textbf{measure universal} if it is universal in the collection of all measurable subsets of $\R$ that have positive Lebesgue measure. One can show, for example using the Lebesgue density theorem, that every finite set is measure universal; this observation is often attributed to H. Steinhaus \cite{Steinhaus}. In 1974, P. Erd\H{o}s asked whether  finite sets are in fact the \textit{only} sets that are measure universal. 

\begin{question}\label{q:erdos}
    Does there exist an infinite set $A \subset \R$ that is measure universal?
\end{question}

In \cite{Erdos74}, Erd\H{o}s writes: ``I hope there are no such sets''. For this reason, we refer to the claim that there are no infinite measure universal sets as the \textbf{Erd\H{o}s similarity conjecture.} This is one of many problems proposed by Erd\H{o}s that remains open to this day\footnote{It is also recorded in https://www.erdosproblems.com/}. There have been several important developments on this problem between 1974 and 1999. More recently, new progress in the last few years has reinvigorated interest in this research area. The goal of this paper is to highlight these recent advancements and provide the reader with an overview of the current status of the problem. 

\medskip

\subsection{Early work on the Erd\H{o}s similarity conjecture.} 
Let us briefly review some of the important early results on the Erd\H{o}s similarity problem. We refer the reader to the excellent paper \cite{Svetic01} of R. Svetic for a comprehensive survey of the progress made between 1974 and 1999. 

It is a simple observation that if $A$ is not measure universal, then any set $B\supseteq A$ cannot be measure universal. Moreover, no unbounded set can be measure universal.  Therefore,  the conjecture can be resolved if one can show that all bounded strictly decreasing sequences with limit 0 are not measure universal.  

\begin{question}\label{q:erdos2}
    Does there exist a decreasing sequence $a_n \to 0$ that is measure universal?
\end{question}

A {\bf sublacunary} sequence  is a decreasing sequence such that $\lim_{n\to\infty}{a_{n+1}/a_n}=1$. Using a direct Cantor set construction, K. Falconer \cite{Falconer84} and S. Eigen \cite{Eigen85} independently proved that sublacunary sequences are not measure universal. 

\begin{theorem}[Eigen, Falconer]\label{th:eigen-falconer}
    Let $a_n \to 0$ be a decreasing sequence. If 
    \begin{align}\label{eq:sublacunary}
        \lim_{n \to \infty} \frac{a_{n+1}}{a_n} = 1
    \end{align}
    then $(a_n)_{n=1}^\infty$ is not measure universal. 
\end{theorem}

No one has been able to construct faster-decreasing sequences that are not measure universal. In particular, determining whether $(2^{-n})_{n=1}^\infty$ is  measure universal has been the major bottleneck of the problem.  It is therefore natural to consider other types of sets beyond decreasing sequences.  J. Bourgain \cite{Bourgain87} showed that a set $X$ is measure universal if and only if there exists $C>0$ such that for all $J>0$, all finite subsets $Y$ of $X$, all ${\bf v}\in\R^J$ and  all continuous $f:{\mathbb T}^J\to\R$,    
\begin{equation}\label{eq-bourgain}
    \int_{{\mathbb T}^J} \inf_{1<\lambda<2} ~\sup_{y\in Y}~ |f({\bf x}+ \lambda y{\bf v})|~ d{\bf x} \leq C \int_{{\mathbb T}^J} |f({\bf x})|~d{\bf x},
\end{equation}
where ${\mathbb T}^J$ is the $J$-dimensional torus.  He used this characterization to show the following

\begin{theorem}[Bourgain]\label{th:bourgain}
If $A_1, A_2, A_3 \subseteq \mathbb{R}$ are infinite, then the sumset $A_1 + A_2 + A_3 = \{a + b + c : a \in A_1, b \in A_2, c \in A_3\}$ is not measure universal. 
\end{theorem}

This result cannot be deduced from Theorem \ref{th:eigen-falconer} as the sets $A_i$ may be sequences of arbitrarily rapid decay. The idea behind Theorem \ref{th:bourgain} is to construct a function that fails (\ref{eq-bourgain}). With a suitable mollification, we may assume $f$ is a random sum of indicator functions of small cubes, and each cube has a probability $\varepsilon$ of being selected. Then $\int|f|  \sim \varepsilon$. When there is an additive structure of three infinite sets, there is a decrease of entropy which forces the left-hand side to be equal to 1 with a very high probability. We refer the readers to T. Tao's article \cite{Tao2021}, where proof of this result is presented as an application of several tools developed by Bourgain. 






It is conceivable that a probabilistic construction---choosing basic intervals independently with a certain probability distribution---could yield an avoiding set for fast-decaying sequences. This strategy was developed by M. Kolountzakis in \cite{Kolountzakis97}. One of his results is the following criterion for a sequence not to be measure universal. 

\begin{theorem}[Kolountzakis]
    If an infinite set $A \subset \mathbb{R}$ contains, for each $n \in \N$, elements $a_1 > a_2 > \ldots > a_n > 0$ such that $-\log(\delta_n) = o(n)$, where
    \begin{align*}
        \delta_n = \min_{i \in \{1, \ldots, n-1\}} \frac{a_i - a_{i+1}}{a_1},
    \end{align*}
    then $A$ is not measure universal.
\end{theorem}
This theorem says that if $A$ contains arbitrary long ``chunks'' of slow-decaying sequences, then $A$ is not measure universal. Clearly, this generalizes Theorem \ref{th:eigen-falconer}. It can also be used to show that sumsets of the form $\{2^{-n^\alpha}\} + \{2^{-n^\alpha}\}$ are not measure universal for any $\alpha \in (0,2)$.  Kolountzakis also established the following almost everywhere solution to the Erd\H{o}s similarity problem. 

\begin{theorem}[Kolountzakis]\label{th:almost-erdos-similarity-kol97}
    For every infinite set $A \subseteq \R$, there exists a measurable set $E_1 \subset [0,1]$ of Lebesgue measure arbitrarily close to 1 such that
    \begin{align*}
        m\big(\{y: \text{there exists $x \in \R$ such that $x + yA \subset E_1$}\}\big) = 0.
    \end{align*}
Moreover, there also  exists a measurable set $E_2$ of positive Lebesgue
measure such that the set
$$
\{(\lambda, t) \in \R^2 : \lambda A+t \subset  E_2\}
$$
has two-dimensional Lebesgue measure zero.
\end{theorem}

Finally, we would be remiss not mention that the conjecture has also inspired a large body of work concerning the existence of various point configurations in fractal sets. For any sequence of sets containing at least three points, T. Keleti \cite{K08} constructed a compact subset of the real line with Hausdorff dimension $1$ and Lebesgue measure zero that contains no affine copy of any of the sets in that sequence. This shows that finite sets cannot be universal in sets of Hausdorff dimension one. Keleti's result was generalized to arbitrarily gauge functions $h$ such that $h(x)/x\to0$ as $x\to 0^+$ as well as in higher dimensions by Yavicoli \cite{Ya2021}.  By assuming certain Frostman and Fourier decay conditions, I. {\L}aba and M. Pramanik \cite{LP09} showed that a large class of fractal sets contain  non-trivial 3-term arithmetic progressions. Meanwhile, P. Shmerkin \cite{S2017} constructed Salem sets of full Hausdorff dimension and Fourier dimension that contain no 3-term arithmetic progression. A recent in-depth study of the detection of patterns in relation to the Fourier dimension can be found in \cite{LP22}.  At the opposite extreme of the Erd\H{o}s problem, one can find small fractal sets containing many patterns: in \cite{DMT60} and \cite{UY2020}, it is shown that there exist perfect sets of real numbers with Hausdorff dimension zero that contain affine copies of all finite sets (and more generally, polynomial copies of finite sets and some countable patterns too). 


\subsection{Overview of this paper.} This paper will survey several novel results obtained in the last few years.  

In Section \ref{section 2}, we introduce a notion of bi-Lipschitz measure universality. This relaxes the requirement of having an affine copy of a sequence to a bi-Lipschitz copy. This notion was first proposed in \cite{Bilipschitz}, where it is proved that all lacunary sequences are indeed bi-Lipschitz measure universal, while sublacunary sequences are still not bi-Lipschitz measure universal (this provides a new proof of Theorem \ref{th:eigen-falconer}). We also indicate that we can study the size of the set of ``Erd\H{o}s points'' of a compact set in order to study the similarity conjecture. Erd\H{o}s points were first considered in \cite{CruzLaiPramanik23}.

The first uncountable set that is not measure universal was discovered in \cite{GallagherLaiWeber23}. In Section \ref{Section 3}, we recast measure universality as a stable sumset problem. As a consequence, we show that Cantor sets of positive Newhouse thickness or positive Hausdorff dimension are not ``full measure universal'', a slightly stronger universality than measure universality. 

In Section \ref{Section 4}, we study ``topological universality''. It turns out that we can completely describe all topologically universal sets. The existence of an uncountable topologically universal set is independent of the $\mathsf{ZFC}$ axioms of set theory.

In Section \ref{sec:large}, we study a variant of the Erd\H{o}s similarity problem ``in the large''; i.e. for unbounded sets. It was first proposed by L. Bradford, H. Kohut and the third-named author \cite{BKM23} and shortly thereafter generalized by M. Kolountzakis and E. Papageorgiou \cite{KolountzakisPapageorgiou23}. We believe that this is the correct dual problem to the original Erd\H{o}s problem. In a recent paper \cite{GMY24}, X. Gao, C. H. Yip, and the third-named author connected this problem to estimates on certain well-studied objects in metric number theory. We give a short survey of this relationship in Section \ref{section 6}.

Unfortunately, there is still no precise statement saying the Erd\H{o}s similarity problem and the Erd\H{o}s similarity problem in the large are actually equivalent. In \cite{BGKMW23}, by slightly improving the universality ``in the large", it was shown that there exists a set containing no geometric sequences decreasing to a Lebesgue density point. We give a short proof in Section \ref{sec:connections}. 

Using probabilistic methods, Kolountzakis recently showed that certain Cantor sets with Newhouse thickness zero are also not measure universal \cite{Kol-cantor-set}. His method appeared in \cite{Kolountzakis97,KolountzakisPapageorgiou23,Kol-cantor-set} has been successfully applied to all different variants of the Erd\H{o}s similarity problem. Readers are invited to consult these articles for details; they are self-contained and well-written.

\section{Decreasing sequences and bi-Lipschitz embedding}\label{section 2}

In this section, we study universality via a bi-Lipschitz perspective originally proposed in \cite{Bilipschitz}.

Recall that $f:\R\to\R$ is a bi-Lipschitz mapping if there exists a constant $L\ge 1$ such that 
$$
L^{-1}|x-y|\le |f(x)-f(y)|\le L|x-y|.
$$
We say that a set $A$ is {\bf bi-Lipschitz measure universal} if for all measurable sets $E$ of positive Lebesgue measure, we can find a bi-Lipschitz map $f$ such that $f(A)\subset E$. Clearly, an affine map is a bi-Lipschitz map.  Although the original Erd\H{o}s similarity problem remains open, a rather sharp result about a bi-Lipschitz variant was obtained in \cite{Bilipschitz}. 

\begin{theorem}[Feng--Lai--Xiong]\label{thm-decreasing}
Let $A = (a_n)_{n=1}^{\infty}$ be a strictly decreasing sequence converging to $0$ and $E$ be a measurable set of positive Lebesgue measure on $\R^1$.  
\begin{enumerate}
    \item Suppose that $\lim_{n\to\infty} \frac{a_{n+1}}{a_n}=1$. Then $A$ is not bi-Lipschitz measure universal.
        \item Suppose that $\limsup_{n\to\infty} \frac{a_{n+1}}{a_n}<1$. Then there exists a bi-Lipschitz map $f: \R\to\R$ such that $f(A)\subset E$.
\end{enumerate}
\end{theorem}

\subsection{Slow-decreasing sequences.} In this subsection we give a proof of Theorem \ref{thm-decreasing} (1). This furnishes a new proof Falconer and Eigen's Theorem \ref{th:eigen-falconer}. (Yet another proof of Theorem \ref{th:eigen-falconer} is available in \cite{Svetic01}.) We first need a lemma. 

\begin{lemma}\label{lem-3.1}
	Let $(a_n)_{n=1}^\infty$ be a strictly decreasing sequence of positive numbers such that
	$$\lim_{k\to \infty}  a_n=0\quad\mbox{and} \quad \lim_{n\to \infty} \frac{a_{n+1}}{a_n}=1.$$
	 Then there exists a subsequence~$(a_{n_k})_{k=1}^\infty$  such that $\lim_{k\to \infty}a_{n_{k+1}}/a_{n_k}=1$ and
	\begin{equation}\label{eq:dis}
	a_{n_k}- a_{n_{k+1}}\leq 2 (a_{n_m}- a_{n_{m+1}})	\quad\mbox{ for all }\; k,m\in \N \mbox{ with }\; k> m.
	\end{equation}
	\end{lemma}
\begin{proof}
For each  $n\in \N$, let
$$
t_n=\sup\{a_p-a_{p+1}:\; p\geq n\}.
$$
Clearly, $(t_n)$ is monotonically decreasing. Since the sequence $(a_n)$ is strictly decreasing with limit $0$, the supremum in the above equality is attainable for each $n$; that is, for each $n$ there exists $p(n)\in \N$ such that $$p(n)\geq n\quad \mbox{ and }\quad t_n=a_{p(n)}-a_{p(n)+1}.$$

Next, we inductively define a subsequence $(n_k)_{k=1}^\infty$ of natural numbers. Set $n_1=1$. Suppose $n_1,\ldots, n_k$ have been defined. Then we define
\begin{equation}
\label{e-e1}
n_{k+1}=\inf\{p\in \N\colon p>n_k,\;  a_{n_k}-a_p\geq t_{n_k}\}.
\end{equation}
Since $a_{n_k}-a_{p(n_k)+1}\geq a_{p(n_k)}-a_{p(n_k)+1}=t_{n_k}$,  it follows from \eqref{e-e1} that
$$n_k<n_{k+1}\leq p(n_k)+1<\infty.$$   Continuing this process, we obtain the sequence $(n_k)_{k=1}^\infty$. By \eqref{e-e1}, $a_{n_k}-a_{n_{k+1}}\geq t_{n_k}$. Moreover, by \eqref{e-e1} and the definition of $t_{n_k}$,
$$
a_{n_k}-a_{n_{k+1}}= (a_{n_{k}}-a_{n_{k+1}-1})+(a_{n_{k+1}-1}-a_{n_{k+1}})\leq t_{n_{k}}+t_{n_{k}}=2t_{n_{k}}.
$$
That is,
\begin{equation}
\label{e-e2}
t_{n_{k}}\leq a_{n_{k}}-a_{n_{k+1}}\leq 2 t_{n_{k}}\quad  \mbox{  for all }k\geq 1.
\end{equation}
It follows that for any $k,m\in \N$ with $k>m$,
$$a_{n_{k}}-a_{n_{k+1}}\leq 2 t_{n_{k}}\leq 2 t_{n_m}\leq 2(a_{n_{m}}-a_{n_{m+1}}).$$

Finally we show that $\lim_{k\to \infty} a_{n_{k+1}}/a_{n_k}=1$, which is equivalent to  $$\lim_{k\to \infty} (a_{n_k}-a_{n_{k+1}})/a_{n_k}=0.$$
Notice that by \eqref{e-e2}, $a_{n_k}-a_{n_{k+1}}\leq 2t_{n_k}=2(a_{p(n_k)}-a_{p(n_k)+1})$,  and
$a_{n_k}\geq a_{p(n_k)}$. Hence
$$
0\leq \frac{a_{n_k}-a_{n_{k+1}}}{a_{n_k}}\leq \frac{2(a_{p(n_k)}-a_{p(n_k)+1})}{a_{p(n_k)}}=2\left(1-\frac{a_{p(n_k)+1}}{a_{p(n_k)}}\right)\to 0,
$$
as desired.
\end{proof}

\begin{proof}[Proof of Theorem \ref{thm-decreasing} (1).]
From Lemma \ref{lem-3.1}, we may assume that 
 \begin{equation}\label{e-e3}
a_n-a_{n+1}\leq 2 (a_m-a_{m+1})\quad  \mbox{ for all } n,m \mbox{ with }n>m.
\end{equation}
Since 
$$
\liminf_{n\to \infty} \frac{a_n-a_{n+1}}{a_n}=1-\limsup_{n\to \infty} \frac{a_{n+1}}{a_n}=0,
$$   we can choose a  strictly increasing sequence $(n_k)_{k=1}^\infty$ of natural numbers such that
$$
\frac{a_{n_k}-a_{n_{k}+1}}{a_{n_k}}\leq k^{-2}4^{-k}\quad \mbox{ for }\;k\geq 1.
$$
Equivalently, we have 
$$
\left(\frac{k}{a_{n_k}}\right)\cdot k (a_{n_k}-a_{n_{k}+1}) <\frac{1}{4^k}\quad \mbox{ for }\;k\geq 1.
$$
Let $\delta_k=k(a_{n_k}-a_{n_k+1})$. We will create a Cantor set whose $k^{th}$ level gaps are of length $\delta_k$ and the gaps are at the position around $j\cdot \frac{k}{a_{n_k}}$ for $j = 0,1,2,\ldots$. To do this, for each $k\geq 1$, let $\ell_k$ be the smallest integer larger than or equal to $ \frac{k}{a_{n_k}}$.   Then  it holds that
$$
\frac{1}{ \ell_k}\leq \frac{a_{n_k}}{k}<\frac{2}{\ell_k}
$$
so that 
\begin{equation}
\label{e-e4}
\ell_k\delta_k\leq \frac{2k}{a_{n_k}}\cdot k(a_{n_k}-a_{n_k+1})\leq  2\cdot 4^{-k}\quad \mbox{ for }k\geq 1.
\end{equation}
Define a sequence $(E_k)_{k=1}^\infty$ of compact subsets of $[0,1]$ by
$$
E_k=[0,1]\setminus \bigcup_{j=0}^{\ell_k}\left(\frac{j}{\ell_k}-\frac{\delta_k}{2},\; \frac{j}{\ell_k}+\frac{\delta_k}{2}\right).
$$
For each $k\in\N$,  $E_k$ is the union of $\ell_k$ disjoint intervals of length $(1-\delta_k\ell_k)/\ell_k$, with a gap of length $\delta_k$ between any two adjacent intervals.
Define $$E=\bigcap_{k=1}^\infty E_k.$$
Then $E$ is a compact set with positive Lebesgue measure. Indeed, using \eqref{e-e4}, 
$$
\mathcal L (E)\geq 1-\sum_{k=1}^\infty\mathcal L\left([0,1]\setminus E_k]\right)\geq 1-\sum_{k=1}^\infty  \ell_k\delta_k\geq 1-\sum_{k=1}^\infty 2\cdot 4^{-k}=\frac13>0.
$$

\medskip

Below we show by contradiction that $(a_n)_{n=1}^\infty$ cannot be embedded into $E$ by a bi-Lipschitz map. Suppose on the contrary that $(a_n)_{n=1}^\infty$ can be embedded into~$E$ by a bi-Lipschitz map~$f:\R\to \R$. Let $b_n=f(a_n)$ for $n\geq 1$ and $b_\infty=\lim_{n\to \infty} b_n$. Then $b_n,\; b_\infty\in E$.  Clearly, $b_\infty=f(0)$, and $(b_n)_{n=1}^\infty$ is strictly monotone increasing or monotone decreasing. Since $f$ is bi-Lipschitz,  there exists a constant $L>1$ such that
\begin{equation*}
\label{e-e5'}
L^{-1} \le \frac{|b_n-b_m|}{a_n-a_m}\le L \quad \mbox{ for all } n, m\in \N, n\neq m.
\end{equation*}
 In particular, this implies that
\begin{equation}\label{e-e5}
	L^{-1} \le \frac{|b_n-b_{n+1}|}{a_{n}-a_{n+1}} \le L	 \quad\text{and}\quad L^{-1} \le \frac{|b_n-b_\infty|}{a_n} \le L.
	\end{equation}
Now fix an integer $k>2L$.
 Then by (\ref{e-e5}) and \eqref{e-e3},  for all $m\ge n_k$,
\begin{equation}
\label{e-e6}
|b_m-b_{m+1}| \le  L (a_m-a_{m+1}) \leq 2L (a_{n_k}-a_{n_k+1})<k (a_{n_k}-a_{n_k+1})=\delta_k.
\end{equation}
As $E_k$ is the union of $\ell_k$ disjoint intervals of length $(1-\delta_k\ell_k)/\ell_k$, with a gap of length $\delta_k$ between any two adjacent intervals,  the sequence $(b_m)_{m=n_k}^\infty$ must be entirely contained in a component interval of $E_k$. This forces that $|b_m-b_\infty|\leq (1-\delta_k\ell_k)/\ell_k$,  Meanwhile, by \eqref{e-e5},
\begin{equation}
\label{e-e7}
|b_{n_k}-b_\infty|\geq \frac{a_{n_k}}{L}> \frac{a_{n_k}}{k}\geq \frac{1}{\ell_k}.
\end{equation}	
This is our desired contradiction and completes the proof. 
\end{proof}

\subsection{Fast-decreasing sequences.} In what follows, we give a proof of Theorem \ref{thm-decreasing} (2), which concerns fast-decreasing sequences. 

\begin{proof}[Proof of Theorem \ref{thm-decreasing} (2)]
Let $E$ be a measurable set of positive Lebesgue measure and $x$ be a Lebesgue density point of $E$. By considering the translation $E-x$, we can, without loss of generality, assume that $x = 0$. By our assumption on $(a_n)_{n=1}^{\infty}$, we can find a $\delta>0$  such that for all $n\ge 1$,
\begin{equation}\label{eq_bound1}
\frac{a_{n+1}}{a_n}<\delta<1.
\end{equation}
Let $\eta$ be a real number such that $\delta<\eta<1$. Define  the interval $I_{n} = \left[\eta a_n,a_n\right]$. The assumptions imply that $I_{n}$ are all disjoint since $a_{n+1}<\delta a_n<\eta a_n$. From the Lebesgue density theorem, 
$$
\frac{m(E^{c}\cap I_n)}{m(I_n)} \le \frac{1}{1-\eta}\cdot \frac{m(E^c\cap [0,a_n])}{a_n} \to 0 \ \mbox{as} \ n\to\infty.
$$
 Hence, we can find $n_0\in\N$ such that when $n\ge n_0$, $E\cap I_n\ne\varnothing.$ We can take $b_n\in E\cap I_n$ and define the map $f(a_n) = b_n$ when $n\ge n_0$. When $n<n_0$, we take $f(a_n) = a'_n$  where $a_n'\in E$ such that $a_n'>a_{n_0}$ and $a_{n}'$ are decreasing. We can extend by linearity between points $(a_{n+1},b_{n+1})$ and $(a_{n},b_n)$ to define a piecewise linear increasing map on $\R$. $f$ is now an increasing function. The map is bi-Lipschitz if and only if there exists a uniform bound in $n$ for all the slopes $\frac{b_n-b_{n+1}}{a_n-a_{n+1}}$.  However, by the fact that $b_n\in I_n$ and (\ref{eq_bound1}), we have
 $$
 \frac{b_n-b_{n+1}}{a_n-a_{n+1}}\le \frac{a_n-\eta a_{n+1}}{a_n-a_{n+1}} \le \max_{x\in [0,\delta]} \frac{1-\eta x}{1-x} = \frac{1-\eta\delta}{1-\delta},
 $$
 and 
 $$
  \frac{b_n-b_{n+1}}{a_n-a_{n+1}}\ge \frac{\eta a_{n}-a_{n+1}}{a_n-a_{n+1}}\ge \min_{x\in[0,\delta]} \frac{\eta -x}{1-x} = \frac{\eta-\delta}{1-\delta},
$$
where we find the maximum and minimum using standard calculus. As $\delta<\eta<1$, the maximum and minimum in the above are all bounded away from $\infty$ and $0$, so this completes the proof.
\end{proof}

With slightly more  work, using the fact that $\eta$ can be chosen as close to 1 as possible,  we can also choose a bi-Lipschitz map $f$ so that $f'(0)=1$, meaning that the map becomes close to an affine map at the limit point.  

Another observation is that we can embed the sequence at every Lebesgue density point of the set $E$. It is well-known that this true for affine copies of \textit{finite} sets. For completeness, we prove this below. Note that this provides another proof of the classical observation of Steinhaus mentioned in the introduction.

\begin{proposition}
    Let $A$ be a finite set and let $E\subset \R$ be a measurable set of positive Lebesgue measure. Then for each Lebesgue density point $x\in E$, there exists $\delta>0$ such that $x+\delta A\subset E.$
\end{proposition}

\begin{proof}
    Let $A =\{a_1<a_2<\cdots< a_m\}$. Without loss of generality, we can assume that $A\subset (0,1]$ with $a_1 = \min A>0$. Moreover, we can assume that $x=0$ is the Lebesgue density point, since otherwise we can consider $E-x$. By the Lebesgue density theorem, whenever $\varepsilon < \frac{a_1}{m}$, we can find $T>0$ such that 
    $$
    m (E\cap [0,t]) \ge (1-\varepsilon)t,
    $$
    whenever $0<t\le T$. 
    Suppose that $\delta A$ is not a subset of $E$ for all $\delta>0$. Then we let 
    $$
    E_i = \{\delta\in (0,T]: \delta a_i\not\in E\} 
    $$
 and we have $\bigcup_{i=1}^m E_i = (0,T]$. Taking Lebesgue measure, there exists $i_0$ such that $m(E_{i_0})\ge T/m$. By our construction, $E\cap a_{i_0} E_{i_0} = \emptyset$, where we write $xE = \{xy: y\in E\}$, and $a_{i_0}E_{i_0}\subset [0,T]$ (since $A\subset[0,1]$). Hence, 
$$
[0,T] \supset (E\cap [0,T])\cup a_{i_0} E_{i_0}
 $$   
and the union is disjoint. Taking Lebesgue measure, we have 
 $$
 T\ge (1-\varepsilon)T + a_{1}\frac{T}{m} = \left(1+\frac{a_1}{m}-\varepsilon\right) T>T
 $$
 by our choice of $\varepsilon$, which is a contradiction. Thus, there must be an affine copy of $A$ at the point $0$.
\end{proof}

For a given infinite set $A$ we define the set of all {\bf Erd\H{o}s points} of a compact set $K$ of positive Lebesgue measure to be 
$$
{\mathcal E}_K = \{x\in K:  (x+tA)\cap K^c\ne \emptyset, \ \forall t\ne 0\}.
$$
The Erd\H{o}s similarity conjecture can be proved if we can show that ${\mathcal E}_K = K$. Indeed, we can also prove the conjecture if we can show that $m ({\mathcal E}_K) >0$ for some $K$. In \cite{CruzLaiPramanik23}, A. Cruz, the second-named author, and M. Pramanik show that for $A = (2^{-n})_{n=1}^{\infty}$, there exists a compact set $K$ of positive measure such that ${\mathcal E}_K $ has Hausdorff dimension 1. 

\section{The Erd\H{o}s similarity problem for uncountable sets}\label{Section 3}
It is clear that proving that all countable sets are not measure universal would answer the Erd\H{o}s similarity conjecture in the affirmative. Given the difficulty of this task, it is natural to ask if all uncountable sets are not measure universal. In particular, it was asked in \cite{GallagherLaiWeber23} if all \textbf{Cantor sets}---i.e. compact, totally disconnected, and perfect subsets of $\R$---are measure universal.
\begin{conjecture}
\label{cantor-set}
    Cantor sets in $\R$ are not measure universal.
\end{conjecture}

While the full answer to this question remains unknown, we will see in this section that positive thickness \cite{GallagherLaiWeber23} and positive dimension \cite{martingales} of a Cantor set are sufficient to show that it is not measure universal. We will also see that Cantor sets are not ``topologically universal''. This is a topological analogue of the measure universality first introduced in \cite{GallagherLaiWeber23} (see also \cite{jung2024interiorcertainsumscontinuous, jung2024topologicalerdhossimilarityconjecture} for more general results). These results support the idea that restricting our attention to Cantor sets in the Erd\H{o}s similarity conjecture is a non-trivial and interesting problem.

\subsection{Universality and sumsets}
We can actually prove a stronger property than measure non-universality. Following \cite{jung2024topologicalerdhossimilarityconjecture}, we introduce the following definition:
a set $X\subset \R$ is said to be \textbf{{full measure universal}} if for every Lebesgue measurable set $F$ with full measure, that is, $m(\R\setminus F)=0$, there exists $\lambda \in \R\setminus \{0\}$ and $t\in \R$ such that $\lambda X + t \subset F$.

Note that if $X$ is not full measure universal, then $X$ is not measure universal. The following proposition shows an equivalence between the Erd\H{o}s problem and a sumset problem. A version of this sumset equivalence was previously discovered in \cite{Jas-95}.


\begin{proposition}
\label{sumset}
Let $X$ be a subset of $\R$ and $\mathcal{N}$ denote the collection of all Lebesgue measure zero sets in $\R$. Then the following are equivalent.\\
(i) $X$ is full measure universal.\\
(ii) For all $M\in \mathcal{N}$, there exists $\lambda \in \R\setminus \{0\}$ and $t\in \R$ such that $(\lambda X + t) \cap M = \varnothing$.\\
(iii) For all $M\in \mathcal{N}$, there exists $\lambda \in \R\setminus \{0\}$ such that $X+\lambda M \neq \R$.
\end{proposition}
\begin{proof}
We first prove (i) $\Rightarrow$ (ii). Suppose $X$ is full measure universal and let $M$ be a measure zero set in $\R$. Then, since $X$ is full measure universal, there exists $\lambda \in \R\setminus \{0\}$ and $t\in \R$ such that $\lambda X + t \subset M^c$. Since $\lambda X + t \subset M^c$ if and only if $\lambda X + t \cap M =\varnothing$, (ii) is true.

Secondly, we prove (ii) $\Rightarrow$ (iii). Suppose (ii) is true and let $M$ be a measure zero set in $\R$. Then, there exists $\lambda' \in \R\setminus\{0\}$ such that $(\lambda' X + t)\cap M =\varnothing$. This implies that for all $x\in X$ and $y \in M$, $\lambda' x + t \neq y$. Thus, for all $x\in X$ and $y\in M$, $\lambda' x - y \neq -t$, i.e., $x-\frac{1}{\lambda'} y \neq -\frac{t}{\lambda'}$. Hence, $X-\frac{1}{\lambda'}M \not\ni -\frac{t}{\lambda'}$. Taking $\lambda = -\frac{1}{\lambda'}$, we see that $X+\lambda M \neq \R$. Thus, (iii) is true.

Lastly, we prove (iii) $\Rightarrow$ (i). Suppose (iii) is true and let $F$ be a set of full measure, i.e., $M=\R\setminus F$ is a measure zero set. Then, since $-M$ is also a measure zero set, there exists $\lambda' \in \R\setminus \{0\}$ such that $X+\lambda' (-M) \neq \R$.  This implies that there exists $t'\in \R$ such that $t'\not\in X+\lambda' (-M)$. Hence, for all $x\in X$ and $y \in M$, $t' \neq x - \lambda' y$, i.e., $\frac{1}{\lambda'} x - \frac{t'}{\lambda'} \neq y$. Thus, letting $\lambda = \frac{1}{\lambda'}$ and $t=-\frac{t'}{\lambda'}$, we see that $(\lambda X + t) \cap M=\varnothing$, or equivalently, $\lambda X + t \subset M^c=F$. Thus, $X$ is full measure universal.
\end{proof}

It is now clear that {\it all countable sets are full measure universal} because if $X$ is countable, then $X+\lambda M$ must have measure zero for all $M\in{\mathcal N}$. Hence, Proposition \ref{sumset} (iii) is satisfied.

In light of Proposition \ref{sumset}, understanding the measure universality of Cantor sets is equivalent to understanding sumsets of the form $X+\lambda M$, where $X$ is a Cantor set and $M$ a measure zero set. In \cite{erdos-kunen-mauldin}, P. Erd\H{o}s, K. Kunen, and R. Mauldin showed a result of a very similar flavor.

\begin{theorem}[Erd\H{o}s--Kunen--Mauldin]\label{ekm}
    If $X$ is a perfect subset of $\R$, then there exists a Lebesgue measure zero set $M\subset \R$ such that $X+M=\R$. 
\end{theorem}

Theorem \ref{ekm} does not prove Conjecture \ref{cantor-set} since we do not know if the same $M$ in the theorem  also satisfies  $X+\lambda M = \R$ for all $\lambda\ne 1$.  This naturally leads us to following conjecture, which was asked in \cite{jung2024topologicalerdhossimilarityconjecture}.

\begin{conjecture}\label{conjecture_P}
    If $X$ is a perfect set in $\R$, then there exists a set $M$ in $\R$ with Lebesgue measure zero such that $X+\lambda M=\R$ for all $\lambda \in \R\setminus \{0\}$. 
\end{conjecture}

Conjecture \ref{conjecture_P} implies that every Cantor set is not full measure universal, which then implies that every Cantor set is not measure universal (Conjecture \ref{cantor-set}). Indeed, it is sufficient to find a measure zero set $M$ such that $(X+\lambda M)^\circ \neq \varnothing$ for all $\lambda \neq 0$, since we can take a countable union of translated copies of $M$ and obtain another measure zero set whose algebraic sum with $X$ is $\R$.

\subsection{Cantor sets with positive Newhouse thickness}
In this subsection, we show that if a Cantor set is ``thick'', then it is too big to be universal. To make this precise, we introduce the notion of Newhouse thickness and the Newhouse gap lemma (Lemma \ref{newhouse-gap-lemma}).

We let $\Sigma^0 = \{\emptyset\}$, $\Sigma^n=\{0,1\}^n$ and $\Sigma^{\ast} = \bigcup_{n=1}^{\infty} \Sigma^n$. We will naturally concatenate $\sigma$ and $\sigma'$ in $\Sigma'$ and denote it by $\sigma\sigma'$.  We say that $\sigma$ is a {\bf descendant} of $\sigma'$ (or $\sigma'$ is an {\bf ancestor} of $\sigma$) if $\sigma = \sigma' \widetilde{\sigma}$ for some $\widetilde{\sigma}$. If $\widetilde{\sigma}$ has length 1, $\sigma$ is a called {\bf child} of $\sigma'$.   

We describe a Cantor set in $\R$  as a binary Cantor set. Let $K$ be a Cantor set and let $I_{\emptyset}(K)$ be the convex hull of $K$. Then the complement of the $K$ in $I_{\varnothing}(K)$  is a countable union of disjoint bounded open intervals, which we call {\bf (bounded) gaps}. We first write them as $G_n$, $n=1,2,\ldots$ with $|G_1|\ge |G_2|\ge \ldots$ and if they are of the same length, we will enumerate them from the leftmost one. We now define 
$$
I_{\varnothing}(K) = I_0(K)\cup U_{\emptyset} (K) \cup I_1(K) \ \mbox{and} \  U_{\emptyset} (K): = G_1,
$$
where we also define by convention that $I_0(K)$ is on the left hand side of $G_1$ and $I_1(K)$ is on the right of $G_1$. 
Suppose that $I_{\sigma}(K)$, $\sigma\in \Sigma^{\ast}$ has been defined. We will let $U_{\sigma}(K)$ to be the largest open interval from the collection of $\{G_n\}_{n\in\N}$ that is in the $I_{\sigma}$. Then we define
$$
I_{\sigma}(K) = I_{\sigma 0}(K)\cup U_{\sigma}(K)\cup I_{\sigma 1}(K).
$$
In this way, the Cantor set $K$ can be represented as 
$$
K  = \bigcap_{n=1}^{\infty} \bigcup_{\sigma\in \Sigma^n} I_{\sigma} (K).
$$




In this setting, the Newhouse thickness of $I_{\sigma}$ is defined to be 
$$
\tau_{\mathcal{N}} (I_{\sigma}) = \min\left\{\frac{|I_{\sigma0}(K)|}{|U_{\sigma}(K)|}, \frac{|I_{\sigma1}(K)|}{|U_{\sigma}(K)|}\right\},
$$
and the \textbf{Newhouse thickness} of the Cantor set $K$ is defined to be 
$$
\tau_{\mathcal{N}}(K) = \inf_{\sigma\in \Sigma^{\ast}} \tau_{\mathcal{N}}(I_{\sigma}). 
$$
Note that the Newhouse thickness is invariant under affine transformations on $\R$. Using the following Newhouse gap lemma (Lemma \ref{newhouse-gap-lemma}), J. Gallagher, the second-named author, and E. Weber \cite{GallagherLaiWeber23} proved that Cantor sets with positive Newhouse thickness are not measure universal.

\begin{lemma}[Newhouse gap lemma]\label{newhouse-gap-lemma}
    Let $K_1$ and $K_2$ be two Cantor sets of $\R^1$. Assume $\tau_{\mathcal{N}}(K_1)\tau_{\mathcal{N}}(K_2)\ge 1$ and $K_1$ is not in any gaps of $K_2$ and vice versa. Then $K_1\cap K_2\ne\emptyset$.  
\end{lemma}

\begin{theorem}[Gallagher--Lai--Weber]
    Cantor sets in $\R$ that have positive Newhouse thickness are not full measure universal, and thus not measure universal.
\end{theorem}
\begin{proof}
    Suppose $X$ is a Cantor set in $\R$ with $\tau(X)>0$. Since $X$ not being full measure universal implies $X$ not being measure universal, by Proposition \ref{sumset}, it suffices to prove that there exists a measure zero set $M$ such that
    \begin{equation}\label{intersection}
        (\lambda X + t) \cap M \neq \varnothing \qquad \forall \lambda \in \R\setminus\{0\} \text{ and } \forall t\in \R.
    \end{equation}    
    Also, note that we may assume $X$ has convex hull $[0,1]$, since affine transformations do not affect the measure universality of $K$.

    We construct the desired measure zero set $M$ as a countable union of affine copies of a Cantor set $K$ with measure zero. Since $\tau(X)>0$, there exists $N\in\N$ such that $N>\frac{1}{\tau(X)}$. Let $K$ be a Cantor set constructed by starting with a closed unit interval $[0,1]$ and repeatedly removing the middle $1/(2N+1)$ closed interval, resulting in a symmetric Cantor set with a contraction ratio $\frac{N}{2N+1}$ and $m(K)=0$. Since for each $\sigma\in \Sigma^n$, $|I_{\sigma 0}|=|I_{\sigma 1}|=\left(\frac{N}{2N+1}\right)^{n+1}$ and $|U_\sigma|= \frac{1}{2N+1}\left(\frac{N}{2N+1}\right)^{n}$, the Newhouse thickness of $K$ is given by $\tau(K)= N$. Define
    \begin{align*}
        M=\bigcup_{(n,l)\in\Z\times \Z} 2^n (K+l)
    \end{align*}
and let $\lambda \in \R\setminus\{0\}$ and $t\in \R$. Since $M$ has measure zero by subadditivity of the Lebesgue measure, it suffices to prove \eqref{intersection}.

Note that there exists a unique $(n,l)\in \Z\times \Z$ such that $|\lambda|\in (2^{n-1}, 2^n]$ and $t\in (l2^n, (l+1)2^n]$. We apply the Newhouse gap lemma (Lemma \ref{newhouse-gap-lemma}) to two Cantor sets $K_1=\lambda X + t$ and $K_2=2^n(K+l)$. To check the conditions of the Newhouse gap lemma, we first note that
\begin{align*}
    \tau(K_1) \tau(K_2)= \tau (X) \tau(K) = \tau (X) \cdot N> 1.
\end{align*}
Next, we show that $K_1$ is not in any gaps of $K_2$ and vice versa. To see this, for each $i=1,2$, let $I_i$ be the convex hull of $K_i$ and $O_i$ be the largest bounded gap of $K_i$. Then, $|I_1|=|\lambda|$, $|O_1|=|\lambda|\cdot |O_X| \leq |\lambda|$, where $|O_X|$ is the largest bounded gap of $X$, $|O_2|=2^{n-1}$ and $|I_2|=2^n$. By our choice of $(n,l)$, we have $|O_1|\leq |I_2|$ and $|O_2|\leq |I_1|$, implying that no gap of $K_1$ contains $K_2$ and vice versa.

Thus, the conditions of the Newhouse gap lemma (Lemma \ref{newhouse-gap-lemma}) are satisfied, and we have its conclusion $K_1\cap K_2 \neq \varnothing$. This proves that $M$ is the desired measure zero set, completing the proof.
\end{proof}

\subsection{Cantor sets with positive Hausdorff dimension}

It is well-known that if a Cantor set $K$ has positive Newhouse thickness, then $K$ has positive Hausdorff dimension. Since Cantor sets with positive Newhouse thickness are not measure universal, it is natural to ask is if Cantor sets with positive Hausdorff dimension are not measure universal. This turns out to be true, as we see below. We would like to thank P. Shmerkin for pointing out this argument to us during the BIRS workshop that led to this proceeding. 

Let $s>0$. Recall that a measure $\mu$ on $\R$ is called a \textbf{Frostman measure} (or an $s$-Frostman measure) if there exists $C>0$ such that
\begin{align*}
    \mu(B(x,r)) \leq  C r^s \quad \forall x\in \R, \ r>0.
\end{align*}
In \cite{martingales}, P. Shmerkin and V. Suomala introduced the notion of \textbf{spatially independent martingales}. The detailed constructions and their properties can be found in their paper. We only use the following result, which is a simpler version of their Theorem 13.1.

\begin{theorem}[]\label{shmerkin}\index{}
If $\nu$ is a $s$-frostman measure on $\R$, then there exists a spatially independent martingale $(\mu_n)$ such that, almost surely, the weak limit $\mu_{\infty}$ of $(\mu_n)$ satisfies that $\mu_\infty * T\nu$ is an absolutely continuous measure with H\"{o}lder continuous density for every affine transformation $T$ in $\R$ and $\dim(supp (\mu_{\infty})) <1$. Moreover, the map $(T,x)\to f_T(x)$ is a H\"{o}lder continuous function, where $f_T(x)$ is the density function of  $\mu_\infty * T\nu$. 
\end{theorem}

\begin{corollary}[]\label{}\index{}
A Cantor set in $\R$ that has positive Hausdorff dimension is not measure universal.
\end{corollary}
\begin{proof}
Let $X$ be a Cantor set in $\R$ with $s= \dim_H (X)>0$. For all $\varepsilon\in (0,s)$, by Frostman's Lemma, we can find  a $(s-\varepsilon)$-Frostman measure $\nu$ on $X$ . By Theorem \ref{shmerkin}, there exists a spatially independent martingale $(\mu_n)$ such that almost surely, $\mu_\infty * T\nu$ is absolutely continuous with H\"{o}lder continuous density for every affine transformation $T$ in $\R$. In particular, the support $\text{supp}(\mu_\infty * T\nu)$ has a nonempty interior. Then,
\begin{align*}
    \varnothing \ne(\text{supp}(\mu_\infty * T\nu))^\circ = (\text{supp} (\mu_\infty)+T(\text{supp}(\nu)))^\circ = (A + TX)^\circ 
\end{align*}
where $A=\text{supp} (\mu_\infty)$, for all affine transformations $T$. By the continuity property of the map $(T,x)\to f_T(x)$, we know that there exists $\delta>0$ such that $(A+TX)^{\circ}\supset I$ where $I$ is an interval of length $\delta>0$ for all $T$ is a neighborhood of the identity map. By taking $M = A+\delta\Z$, we see that $\lambda X+M = \R$ for all $\lambda$ in the neighborhood of 1. By considering union of countably many dilates of $M$, we can make $\lambda X+M = \R$ hold for all non-zero $\lambda$.  Multiplying by a constant, we find that $X+\lambda M = \R$. As the support of $\mu_{\infty}$ has measure zero, the set $M$ also has measure zero.  This implies that $X$ is not full measure universal by Proposition \ref{sumset}.
\end{proof}

\section{Cantor sets and topological universality}\label{Section 4}
\subsection{Topological universality} In \cite{GallagherLaiWeber23}, a topological version of the Erd\H{o}s similarity conjecture was considered. A subset $X$ of $\R$ is said to be \textbf{topologically universal} if for every dense $G_\delta$ subset $G$ of $\R$, there exist $\lambda \in \R\setminus \{0\}$ and $t\in \R$ such that $\lambda X + t \subset G$. While the measure universality of Cantor sets remains open, it was proven in the same paper that all Cantor sets in $\R^d$ are not topologically universal. In this subsection, we see how a recent tool (known as the containment lemma) introduced in \cite{jung2024interiorcertainsumscontinuous} gives an alternative proof of this result on $\R^1$.



\begin{lemma}[Containment lemma]
\label{containment-lemma}
Let $K$ be a Cantor set in $\R$. Suppose that $\widetilde{K}$ is a  Cantor set in $\R$ with $I_{\emptyset}(K) \subset I_{\emptyset}(\widetilde{K})$ and \begin{equation}\label{eq_gap-length}
\max\{|U_{\sigma}(\widetilde{K})|:\sigma\in\Sigma^n\}<\min \{ |U_{\sigma}(K)|:\sigma\in \Sigma^n \} \quad \forall n\in\N.
\end{equation}
Then, $K\cap \widetilde{K}\neq \varnothing.$
\end{lemma}
\begin{proof}
Since $K$ and $\widetilde{K}$ are compact, it suffices to show that there exists a sequence $(\alpha_n, \widetilde{\alpha}_n)\in K\times \widetilde{K}$ such that $\lim_{n\to\infty} |\alpha_n-\widetilde{\alpha}_n|=0$. This follows from the next claim:

\textit{Claim.} For each $n\in \mathbb{N}$, there exists $\sigma_n, \sigma_n '\in \Sigma^n$ where $\sigma_n$ is a child of $\sigma_{n-1}$ and $\sigma_n'$ is a child of $\sigma_{n-1}'$, such that $I_{\sigma_n}(K)\subset I_{\sigma_n'}(\widetilde{K})$.

Assuming the claim, we can just take $\alpha_n\in K\cap I_{\sigma_n}(K)$ and $\alpha_n'\in I_{\sigma_n'}(\widetilde{K})$. Since $|I_{\sigma_n'}(K)|\to 0$ as $n\to\infty$ and $|\alpha_n-\alpha_n'|\le |I_{\sigma_n'}(K)|$, this finishes the proof. 

We now prove the claim by induction. Note that the base case is true since $I_\varnothing (K)\subset I_\varnothing (\widetilde{K})$ by assumption. For the induction hypothesis, suppose that there exist $\sigma_n$ and $\sigma_n'$ such that $I_{\sigma_n}(K)\subset I_{\sigma'_n}(\widetilde{K})$ where $\sigma_n$ and $\sigma'_n$ form a chain of children up to $n$. We now proceed to the induction step of $n+1$.

Let us write $I_{\sigma_n}(K) = [a_0,b_0]$ and $I_{\sigma'_n}(\widetilde{K}) = [c_0,d_0]$. Then let us  denote the next children  $I_{\sigma_{n}0}(K)$, $I_{\sigma_{n}1}(K)$, $I_{\sigma'_{n}0} (\widetilde{K})$, $I_{\sigma'_{n}1} (\widetilde{K})$ respectively by  $[a_0,a]$,$[b,b_0]$, $[c_0,c]$ and $[d,d_0]$,  so that  
$$
{\color{black} I_{\sigma_{n}}(K) =[a_0,a]\cup U_{\sigma_n}(K)\cup [b,b_0], \ I_{\sigma'_{n}}(\widetilde{K}) =[c_0,c]\cup U_{\sigma_n'}(\widetilde{K})\cup [d,d_0].}
$$
By the induction hypothesis $c_0\le a_0$ and $b_0\le d_0$, we have
\begin{align*}
 {\color{black}   d-c=|U_{\sigma_n}(\widetilde{K})| < \min \{ |U_{\sigma_{n}}(K)|:\sigma_{n}\in \Sigma^{n} \} \leq |U_{\sigma_{n}}(K)|=b-a}.
\end{align*}
{\color{black} This means that we must have $a<c$ or $d<b$. In the first case, $[c_0,c]\supset [a_0,a]$ meaning that $I_{\sigma'_{n}0}(\widetilde{K})\supset I_{\sigma_{n}0}({K})$, while in the second case, we have $I_{\sigma'_{n}1}(\widetilde{K})\supset I_{\sigma_{n}1}({K})$. This shows that the claim hold true for $n+1$ and hence completes the proof.}
\end{proof}

\begin{theorem}[Jung--Lai]\label{cor-perturbation}
   If $K$ is a Cantor set in $\R$, then there exists a Cantor set $\widetilde{K}$ in $\R$ and $\delta >0$ such that $K\cap (\lambda \widetilde{K}+t) \neq \varnothing$ for all $\lambda\in \left(\frac{1}{1+\delta}, 1+\delta\right)$ and for all $t\in (-\delta, \delta)$. In particular, $K$ is not topologically universal.
\end{theorem}
\begin{proof}
We take the Cantor set $\widetilde{K}$ such that $I_{\emptyset}(\widetilde{K})$ strictly contain $I_{\emptyset}(K)$ and satisfying (\ref{eq_gap-length}) with right hand side replaced by $1/2\cdot \min\{|U_{\sigma}(K)|: \sigma\in\Sigma^n\}$. Then, $K\cap \widetilde{K}\ne\emptyset$. Note that $I_{\sigma}(\lambda K+t) = \lambda I_{\sigma}(K)+t$. By choosing $\delta$ sufficiently small, the conditions $I_{\emptyset}(\lambda K+ t) \subset I_{\emptyset}(\widetilde{K})$ and (\ref{eq_gap-length}) holds true with $K$ replaced by $\lambda K+t$ for all $\lambda\in \left(\frac{1}{1+\delta}, 1+\delta\right)$ and $t\in (-\delta,\delta)$.  The containment lemma (Lemma \ref{containment-lemma}) now implies our first desired conclusion.

To prove that $K$ is not topologically universal, define
\begin{align*}
    M=\bigcup_{(a,b)\in \Q\times \Q} (a\widetilde{K}+b).
\end{align*}
Note that for each $\lambda\in \R\setminus \{0\}$ and $t\in \R$, there exists $(a,b)\in \Q\times \Q$ such that $\lambda^{-1} a \in \left(\frac{1}{1+\delta},1+\delta\right)$ and $\lambda^{-1}b -t \in (-\delta, \delta)$. Thus, by the first part of the theorem, we have
\begin{align*}
    K\cap \left( \lambda^{-1} \left( (a\widetilde{K}+b)-t\right) \right) \neq \varnothing,
\end{align*}
which is equivalent to $\left(\lambda K + t \right)\cap (a\widetilde{K}+b)\neq \varnothing$. This implies that for every $\lambda \in \R\setminus \{0\}$ and $t\in \R$, $(\lambda K + t) \cap M \neq \varnothing$, i.e., $\lambda K + t \not\subset M^c$. By Baire category theorem, $M$ is a nowhere dense $F_\sigma$ set, implying that $M^c$ is a dense $G_\delta$ set. Therefore, $K$ is not topologically universal.
\end{proof}

Unfortunately, the Cantor sets produced in Theorem \ref{cor-perturbation} have positive Lebesgue measure, so we cannot conclude anything about measure non-universality from this approach. 

\subsection{Connection to descriptive set theory} Topologically universal sets can be completely classified via the notion of strong measure zero sets. A set $E\subset \R$ is called {\bf strong measure zero} if for all sequences $(\varepsilon_n)_{n=1}^{\infty}$, we can find intervals $I_n$ covering $E$ and $|I_n|\le \varepsilon_n$.  In \cite{jung2024topologicalerdhossimilarityconjecture}, it was shown that a set is topologically universal if and only it is of strong measure zero. The following is a famous conjecture from descriptive set theory:

{\bf  Borel conjecture $({\mathsf{BC}})$}:  a strong measure zero set must be countable. 

It is known that ${\mathsf{BC}}$ is independent of the ${\mathsf{ZFC}}$ axioms of set theory \cite{Sierpinski, Laver}. Hence, the existence of uncountable topologically universal sets cannot be proved or refuted within ${\mathsf{ZFC}}$. From measure-category duality, there is also a notion of {\bf strongly meager sets}, which are sets $E$ such that $E+M\ne \R$ for all $M\in{\mathcal N}$, the collection of Lebesgue measure zero sets. This leads us to the

{\bf Dual Borel conjecture $({\mathsf{dBC}})$}:  a strongly meager set must be countable. 

It is known that $({\mathsf{dBC}})$ is independent of ${\mathsf{ZFC}}$ as well \cite{Carlson}. We refer interested readers to \cite{jung2024topologicalerdhossimilarityconjecture} for details and other references. 

We conclude Sections \ref{Section 3} and \ref{Section 4} with the following conjecture:

\begin{conjecture}\label{conjecture_Q} \textcolor{white}{Conjecture.}
\begin{enumerate}
\item A set is full measure universal if and only if it is strongly meager.
\item If $X$ is a perfect set in $\R$, then $X$ is not full measure universal.
\end{enumerate}
\end{conjecture}
If (1) is true, then the existence of uncountable full measure universal sets will be independent of ${\mathsf{ZFC}}$. Note that (2) is a restatement of Conjecture \ref{conjecture_P}. Theorem \ref{ekm} shows that all perfect sets are not strongly meager, but we do not know if this is the case for full measure universal sets. 



\section{A variant of the Erd\H{o}s similarity conjecture ``in the large''}\label{sec:large}

In this section, we discuss a variant of the Erd\H{o}s similarity conjecture that may be regarded as a ``dual version'' for unbounded sets. Let $0\le p<1$. A measurable set $E \subseteq \R^n$ is $p$-\textbf{large}  if for every $k \in \Z^n$, 
\begin{align*}
    m(E \cap (k + [0,1]^n)) \geq p. 
\end{align*}
For each $p$, we denote the collection of $p$-large sets by $\mathscr{R}(p)$. A set $A \subseteq \R^n$ is \textbf{universal in} $\mathscr{R}(p)$ if for each $p$-large set $E$, there exists an affine transformation mapping $A$ to a subset of $E$.

The notion of $p$-largeness first appeared in \cite{BKM23}, where L. Bradford, H. Kohut, and the third-named author showed that for each $0 \leq p < 1$, there exists a $p$-large set that does not contain any infinitely long arithmetic progression; i.e. that does not contain any sequence $(x + yn)_{n = 1}^\infty$ for any $x, y \in \R^{n}$ where $y \neq 0$. M. Kolountzakis and E. Papageorgiou, in their paper \cite{KolountzakisPapageorgiou23}, recast this result as saying that no increasing \textit{linear} sequence $a_n \to \infty$ is universal in $\mathscr{R}(p)$ for any $0 \leq p < 1$. They then ask the following question. 

\begin{question}\label{q:large-erdos}
    Does there exist an increasing sequence $a_n \to \infty$ that is universal in $\mathscr{R}(p)$ for some $0 \leq p < 1$?
\end{question}

Question \ref{q:large-erdos} and the Erd\H{o}s similarity problem are similar in spirit. The former concerns the universality of increasing sequences $a_n \to \infty$ among sets occupying an arbitrarily large proportion of Euclidean space, and the latter concerns the universality of decreasing sequences $a_n \to 0$ among sets occupying an arbitrarily small proportion of Euclidean space. For this reason, Kolountzakis and Papageorgiou describe this question as an ``Erd\H{o}s similarity problem in the large''. 

Question \ref{q:large-erdos} is also related to a substantial body of work in analysis concerning the existence of geometric structures in measurable sets $E \subseteq \R$ such that
\begin{align*}
    \delta(E) \coloneq \limsup_{r \to \infty} \frac{m\big(E \cap B(0; r)\big)}{m\big(B(0; r)\big)} > 0 \quad \text{or} \quad \delta_B(E) \coloneq \limsup_{r \to \infty} \sup_{x \in \R^n} \frac{m\big(E \cap B(x; r)\big)}{m\big(B(0; r)\big)} > 0.
\end{align*} 
We say that $E$ has positive \textbf{upper density} if $\delta(E) > 0$ and positive \textbf{upper Banach density} if $\delta_B(E)>0$. Typically, one asks if such sets must contain ``all sufficiently large copies'' of a prescribed geometric configuration. That is, given a finite set $A \subset \R^n$ and a set $E \subseteq \R^n$ with $\delta_B(E) > 0$, one asks if there exists a threshold $y_0 > 0$ such that for each $y > y_0$, there exists an affine transformation of the form $u \mapsto x + yu$ mapping $A$ to a subset of $E$. In \cite{Szekely}, L. Székely asked if every set with positive upper Banach density in $\R^2$ must contain all sufficiently large copies of $A = \{0,1\}$. This question was answered affirmatively in \cite{SzekelyBourgain, SzekelyMarstrand, SzekelyFKW, SzekelyQuas}. Since then, mathematicians have considered other variants of this problem, replacing $A$ with more general finite sets \cite{LM2, LM5, LM3, DG1, LM4, LM1}, equipping $\R^n$ with more general metrics \cite{Sz-Kolountzakis, CMP, DK3, DK1, DK2}, or seeking quantitative bounds on $\delta(E)$ and $\delta_B(E)$ \cite{FKY}. Question \ref{q:large-erdos} extends this line of inquiry to \textit{infinite} sets $A$. We make this precise in Section \ref{sec:connections}. 

In what follows, we survey current progress on Question \ref{q:large-erdos}. For simplicity, we restrict our attention to the one-dimensional setting. 

\subsection{Current progress on Question \ref{q:large-erdos}}\label{sec:results-large}

The first partial answer to Question \ref{q:large-erdos} is the aforementioned theorem from \cite{BKM23}. We restate it here for completeness. 

\begin{theorem}[Bradford--Kohut--Mooroogen]\label{th:bkm}
    Let $a_n \to \infty$ be an increasing sequence. If $(a_n)_{n = 1}^\infty$ is linear, then for each $0 \leq p < 1$ there exists a $p$-large set containing no affine copy of $(a_n)_{n = 1}^\infty$. 
\end{theorem}

The main result of Kolountzakis and Papageorgiou's paper \cite{KolountzakisPapageorgiou23} extends the above theorem to sequences of \textbf{subexponential} growth; i.e. where $\log a_n = o(n)$. Roughly speaking, the faster the growth of a sequence, the more spread out its points (as $n \to \infty$), and therefore the harder it is to construct a set that omits a point of that sequence. Said differently, the points of a slower sequence are less ``sparsely distributed'', so it is easier to show that at least of them escapes our set.

\begin{theorem}[Kolountzakis--Papageorgiou]\label{th:kolountzakispapageorgiou23}
    Let $a_n \to \infty$ be an increasing sequence. If $\inf(a_{n+1} - a_n) \geq 1$ and $\log{a_n} = o(n)$, then for each $0 \leq p < 1$ there exists a $p$-large set containing no affine copy of $(a_n)_{n = 1}^\infty$. 
\end{theorem}

Kolountzakis and Papageorgiou's proof is probabilistic. It uses ideas akin to those first developed by Kolountzakis to study the Erd\H{o}s similarity conjecture \cite{Kolountzakis97}.

Theorem \ref{th:kolountzakispapageorgiou23} is applicable to a number of natural sequences. For example, we can take $(P(n))_{n=1}^\infty$ where $P$ is a monic polynomial. More generally, if $(a_n)_{n=1}^\infty$ is \textbf{sublacunary} (i.e. $a_{n+1}/a_n \to 1$) then it is subexponential, and we obtain the following corollary, which is analogous to Theorem \ref{th:eigen-falconer} on the Erd\H{o}s similarity conjecture. 

\begin{corollary}\label{th:falconer-in-the-large}
    Let $a_n \to \infty$ be an increasing sequence. If $a_{n+1}/a_n \to 1$, then for each $0 \leq p < 1$ there exists a $p$-large set containing no affine copy of $(a_n)_{n =1}^\infty$.
\end{corollary}

It is currently not known if an arbitrary sequence of exponential growth can be universal in $\mathscr{R}(p)$ for any $0 \leq p < 1$. In particular, the following question is open. 

\begin{question}\label{q:large-erdos-2^n}
    Is the sequence $(2^n)_{n = 1}^\infty$ universal in $\mathscr{R}(p)$ for some $0 \leq p < 1$?
\end{question}

The most recent work on Question \ref{q:large-erdos} is due to X. Gao, the third-named author, and C.H. Yip. In \cite{GMY24}, they prove new sufficient conditions for sequences to be non-universal in $\mathscr{R}(p)$. Their work uses techniques from metric number theory. 

Recall that each $x \in \R$ can be expressed as $x = \lfloor x \rfloor + \langle x\rangle$, where $\lfloor x \rfloor$ is the largest integer no greater than $x$. We call $\lfloor x \rfloor$ the \textbf{integer part} and $\langle x \rangle$ the \textbf{fractional part} of $x$. A set $A \subseteq \R$ is \textbf{dense modulo $1$} if $\{\langle x \rangle : x \in A\}$ dense in $[0,1]$. Slightly abusing terminology, we say that a sequence is dense modulo $1$ if the set of its elements is dense modulo $1$.

\begin{theorem}[Gao--Mooroogen--Yip]\label{thm:packing}
     Let $a_n \to \infty$ be an increasing sequence. If the packing dimension of the set 
     \begin{align}\label{eq:gmy-except}
         \mathbf{E} = \{y \in \mathbb{R}: (ya_n)_{n=1}^\infty \text{ is not dense modulo } 1\}
     \end{align}
     satisfies $\operatorname{dim}_P(\mathbf{E}) < 1/2$, then for each $0 \leq p < 1$ there exists a $p$-large set containing no affine copy of $(a_n)_{n =1}^\infty$.
\end{theorem}

The set in \eqref{eq:gmy-except} is a well-studied object in metric number theory. There are several results in this field that quantify the size  of $\mathbf{E}$ under various assumptions on $(a_n)_{n=1}^\infty$. We summarize briefly as follows:

\begin{enumerate}
    \item If $(a_n)_{n=1}^\infty$ is sublacunary, M. Boshernitzan  shows in \cite{Boshernitzan94} that $\operatorname{dim}_P(\mathbf{E}) = 0$. Boshernitzan states his result in terms of Hausdorff dimension, but his proof yields this stronger estimate  (See \cite[Theorem B]{GMY24} for details).   Applying this result to Theorem \ref{thm:packing} gives a new proof that sublacunary sequences are not universal in $\mathscr{R}(p)$ for $0 \leq p < 1$ (Corollary \ref{th:falconer-in-the-large}).
    
    \medskip
    
    \item  A. Pollington and B. de Mathan \cite{Pollington79, Mathan80} showed that $\operatorname{dim}_P(\mathbf{E}) = 1$ for any lacunary sequence (i.e. any sequence such that $a_{n+1}/a_n \to 1 + \eta$ for some $\eta > 0$) so Theorem \ref{thm:packing} cannot be applied to deduce the non-universality of such sequences. 
\end{enumerate}

Note however that there exist sequences that are neither lacunary nor sublacunary, so the above result may be applied to some of those cases. The next corollary furnishes some examples.

\begin{corollary}\label{thm:b-dense}
Let $a_n \to \infty$ be an increasing sequence. If $D = \{\lfloor a_n \rfloor : n \in \Z\}$ has positive upper Banach density; i.e.
\begin{align}\label{eq:gmy-dense}
    \limsup_{n \to \infty} \max_{h \in \N} \frac{\#D \cap\{h + 1, h + 2, \ldots, h + n\}}{n} > 0,
\end{align} 
then the associated set $\mathbf{E}$ defined as in \eqref{eq:gmy-except} is countable and hence it has packing dimension zero. Consequently, for each $0 \leq p < 1$, there exists a $p$-large set containing no affine copy of $(a_n)_{n=1}^\infty$.
\end{corollary}

Consider the sequence $(a_n)_{n=1}^\infty$ obtained by arranging the terms of $\{f(n)+m: m,n \in \N, m \leq n\}$ in increasing order, where $f: \mathbb{N} \to \mathbb{R}$ is an arbitrary increasing function. If $f$ grows sufficiently quickly (e.g. if $f(n) = 2^n$), this sequence is neither lacunary nor sublacunary, but corollary \ref{thm:b-dense} shows that it is nonuniversal in $\mathscr{R}(p)$ for every $0 \leq p < 1$. In fact, by choosing $f$ to grow sufficiently quickly, we can make $(a_n)_{n=1}^\infty$ grow as fast as we wish. In particular, we can ensure that $\log{a_n} \neq o(n)$ so the nonuniversality of this sequence cannot be deduced from Theorem \ref{th:kolountzakispapageorgiou23}.

To prove Corollary \ref{thm:b-dense}, the authors of \cite{GMY24} prove a new result on the distribution of sequences modulo $1$, extending work of Y. Amice \cite{Amice}, J.P. Kahane \cite{Kahane64}, and J. Haight \cite{Haight88}. 

\subsection{Some progress for lacunary sequences.} Let us now discuss Question \ref{q:large-erdos-2^n} for some specific lacunary sequences. Although none of the results mentioned so far are applicable to the sequence $(2^n)_{n=1}^\infty$, the authors of \cite{GMY24} prove that certain exponential sequences $(b^n)_{n=1}^\infty$ are not universal in $\mathscr{R}(p)$ for a restricted range of $p$ that depends on the algebraic properties of the base $b$.

\begin{theorem}[Gao--Mooroogen--Yip]\label{th:special-exp}
Let $b \in \R$ be an algebraic number, $f \in \Z[x]$ its minimal polynomial, and $G = \{\sum_{j=0}^m a_j x^j \in \R[x]: a_0 = 1 \text{ or } a_m = 1\}$. Define 
\begin{align*}
    \ell(b) = \inf_{g \in G} L(fg)
\end{align*}
where $L(fg)$ is the sum of the absolute values of the coefficients of the polynomial $fg$. For each $0 \leq p < \ell(b)$, there exists a $p$-large set containing no affine copy of $(b^n)_{n=1}^\infty$.
\end{theorem}

The special case $b = 2$ partially answers Question \ref{q:large-erdos-2^n}. 

\begin{corollary}\label{cor:2nlarge}
For each $0 \leq p < 1/2$ there exists a $p$-large set containing no affine copy of $(2^n)_{n=1}^\infty$.
\end{corollary}

\section{Arguments using metric number theory} \label{section 6}

Several of the results on Question \ref{q:large-erdos} employ tools from metric number theory. In this section we explain why it is natural to consider these tools. To begin, let $a_n \to \infty$ be an increasing sequence. It is known that if $\inf (a_{n+1} - a_n) \geq 1$, then
\begin{align*}
         \mathbf{E} = \{y \in \mathbb{R}: (ya_n)_{n=1}^\infty \text{ is not dense modulo }1\}
\end{align*}
has Lebesgue measure zero (see for example \cite[Section 1.4]{KN74}). Using this fact, it is easy to show that for each $0 \leq p < 1$, there exists a $p$-large set not containing \textit{almost every} affine copy of $(a_n)_{n=1}^\infty$. 
     
\begin{theorem}\label{th:almostlarge}
    Let $a_n \to \infty$ be an increasing sequence. For every $0 \leq p < 1$, there exists a $p$-large set containing no sequence of the form $(x + ya_n)_{n=1}^\infty$ for every $x \in \R$ and (Lebesgue) almost every $y \in \R$.
\end{theorem}

\begin{proof}
By passing subsequences if necessary, we may assume that $\inf (a_{n+1} - a_n) > 1$.  Given $0 \leq p < 1$, consider the $p$-large set 
    \begin{align}\label{eq:bkm-basic}
        E_p = \{x \in \R : 0 \leq \langle x\rangle < p \}.
    \end{align} 
    If $x \in \R$ and $y \in \R \setminus \mathbf{E}$, then the sequence $(x + ya_n)_{n=1}^\infty$ is dense modulo $1$ and must therefore contain elements $x+y a_{k}$ whose fractional part larger than $p$. $E_p$ cannot contain such elements from definition, so $(x + ya_n)_{n=1}^\infty$  cannot be a subset of $E_p$. Since $\R \setminus \mathbf{E}$ has full Lebesgue measure, the proof is complete. 
\end{proof}

Theorem \ref{th:almostlarge} bears a striking resemblance to Kolountzakis's Theorem \ref{th:almost-erdos-similarity-kol97} on the Erd\H{o}s similarity conjecture. In light of this result, it seems reasonable to conjecture that every sequence $a_n \to \infty$ is not universal in $\mathscr{R}(p)$ for every $0 \leq p < 1$.

To show that a sequence $(a_n)_{n=1}^\infty$ is not universal in $\mathscr{R}(p)$, we must construct a $p$-large set not containing \textit{any} affine copy of that sequence. The proofs of Theorems \ref{th:bkm}, \ref{thm:packing}, and Corollary \ref{thm:b-dense} share a common strategy: they exploit the fact that under additional assumptions on $(a_n)_{n=1}^\infty$, the sets $\mathbf{E}$ are small---smaller than merely being of measure zero. In these cases, constructions like \eqref{eq:bkm-basic} provide simple $p$-large sets that do not contain the vast majority of affine copies of $(a_n)_{n=1}^\infty$, and it remains to find a construction that works for the ``few'' remaining affine copies with dilation parameter in $\mathbf{E}$. To illustrate this strategy, let us sketch a proof of Theorem \ref{th:bkm}.

\begin{proof}[Proof of Theorem \ref{th:bkm}] For simplicity, we only construct a $1/2$-large set not containing any affine copy of $(n)_{n=1}^\infty$.  This construction can  be adapted to produce a $(n-2)/n$-large set with this property for any $n \in \N$. For sufficiently large $n$, this proves Theorem \ref{th:bkm}.

Partition $[0,1]$ into subintervals $I_j = [j/4, (j+1)/4]$, where $j \in \{0, 1, 2, 3\}$. Set
\begin{align*}
    (b_k)_{k = 0}^\infty = (0,1,2,\overbrace{0,0,1,1,2,2}^\text{Repeat digits twice.},\underbrace{0,0,0,1,1,1,2,2,2}_\text{Repeat digits thrice.}, \overbrace{0,0,0,0, 1,1,1,1, 2,2,2,2}^\text{Repeat digits four times.}, \ldots)
\end{align*}
and define
\begin{align*}
    E = \bigcup_{k = 0}^\infty \big([0,1] \setminus (I_{b_k} \cup I_3)\big) + k.
\end{align*}

This set is $1/2$-large since for each $k$ we keep exactly two intervals of the form $I_j + k$. We claim that $E$ contains no sequence $(x + yn)_{n=1}^\infty$ for any real numbers $0 \leq x < 1$  and $y > 0$. It is easy to check that the exceptional set ${\bf E}$ in (\ref{eq:gmy-except}) is exactly $\Q$ for the sequence $(n)_{n=1}^{\infty}$. If $y$ is irrational, then $(x + yn)_{n=1}^\infty$ is dense modulo $1$. Thus, $y \in \R\setminus \mathbf{E}$, and the claim follows from the fact that $E \subset E_\frac{1}{3}$ where $E_\frac{1}{3}$ is as in \eqref{eq:bkm-basic}.

Suppose now that $y \in \Q = \mathbf{E}$. Choose a subsequence $(x + y'n)_{n=1}^\infty$ where $y'$ is an integer. The terms of this subsequence are all equal modulo $1$, so there exists {exactly one} $j_0 \in \{0, 1, 2\}$ such that the entire subsequence belongs to intervals of the form $I_j + k$. If $E$ contains $(x + y'n)_{n=1}^\infty$, then 
\begin{align*}
    (x + y'n)_{n=1}^\infty \subset \bigcup_{k = 1}^\infty (I_{j_0} + k) \cap E = \bigcup_{k \text{ s.t } b_k = {j_0}} (I_{j_0} + k).
\end{align*}
The set on the right-hand side is a union of disjoint intervals. Since the separation between its consecutive intervals becomes arbitrarily large as $k \to \infty$, it cannot contain any unbounded sequence of equally-spaced points. This is a contradiction.
\end{proof}

The above proof exploits the fact that the set $\mathbf{E}$ is countable. The next lemma gives a more systematic way of obtaining a set avoiding all affine copies of a given sequence, provided that the dilation parameters of the affine maps are restricted to a countable set. The proof (omitted here) uses a similar construction to the one presented above.

\begin{lemma}[Gao--Mooroogen--Yip]\label{lem:construction2}
    Let $a_n \to \infty$ be an increasing sequence of real numbers and $C \subset (0, \infty)$ a countable set. For each $0 \leq p < 1$, there exists a $p$-large set containing no sequence of the form $(x + ya_n)_{n = 1}^\infty$ where $x \in \R$ and $y \in C$.
\end{lemma}

It is natural to ask how one can construct a $p$-large set avoiding all affine copies of a given sequence if its associated set $\mathbf{E}$ is uncountable. The next result provides one possible method.

\begin{lemma}[Gao--Mooroogen--Yip]\label{lem:construction-lemma-gmy24}
    Let $a_n \to \infty$ be an increasing sequence. If
    \begin{align}\label{eq:gmy-lem}
        \mathbf{EE}^{-1} = \left\{ \frac{x}{y} : x,y  \in \mathbf{E}, y \neq 0\right\} \neq \R,
    \end{align}
   then for each $0 \leq p < 1$, there exists a $p$-large set containing no affine copy of $(a_n)_{n = 1}^\infty$. 
\end{lemma}

\begin{proof}
    Fix $0 \leq p < 1$ and $y \in \mathbb{R} \setminus (\mathbf{EE}^{-1})$ with $y > 1$.  This is always possible since $\mathbf{E}=-\mathbf{E}$ and therefore $\mathbf{EE}^{-1}$ is symmetric about $0$. Let $0 \leq p' < 1$ be a parameter to be determined later and define
    \begin{align*}
        D = \left\{x \in \R : 0 \leq \langle x \rangle < 1- p'\right\} \quad \text{and} \quad E = D \cap yD.
    \end{align*}
    Arguing as in Theorem \ref{th:almostlarge}, we see that $D$ does not contain any sequence $(x + za_n)_{n=1}^\infty$ where $x \in \R$ and $z \in \R \setminus \mathbf{E}$. Since $E \subset D$, the same is true of $E$. Suppose that there exist some $x \in \R$ and $z \in \mathbf{E} \setminus \{0\}$ such that $(x + za_n)_{n=1}^\infty \subset E$. Then $(xy^{-1} + zy^{-1}a_n)_{n=1}^\infty \subset D$ and therefore $zy^{-1} \in \mathbf{E}$. But now
    \begin{align*}
        y = \frac{z}{zy^{-1}} \in \mathbf{EE}^{-1}
    \end{align*}
    contradicting our choice of $y$. We conclude that $E$ contains no affine copy of $(a_n)_{n=1}^\infty$.
    
    We now show that $E$ is $p$-large. Fix an interval of unit length $I\subset \R$ and choose an interval of unit length $J\subset \R$ such that $y^{-1}I \subset J$. Then
    \begin{align*}
        m( I\setminus E ) &\leq m( I \setminus D ) + m( I \setminus yD) \\
        &= m( I \setminus D ) + y m\big( (y^{-1} I) \setminus D \big) \\
        &\leq (1+y) \max\{m(I \setminus D), m(J\setminus D)\}.
    \end{align*}
    We can make the right-hand side as small as desired by choosing $p' \to 1$. In particular, we can ensure that $m(I\setminus E) \leq 1- p$ for every interval of unit length $I$, which implies that $E$ is $p$-large.
\end{proof}


The proofs of Theorem \ref{thm:packing} and Corollary \ref{thm:b-dense} both  follow from Lemma \ref{lem:construction-lemma-gmy24}. For Theorem \ref{thm:packing}, it is a standard exercise in fractal geometry that if $\operatorname{dim}_P\mathbf{E} < 1/2$ then the Hausdorff dimension of $\mathbf{EE}^{-1}$ is strictly smaller than $1$, and therefore $\mathbf{EE}^{-1} \neq \R$. Corollary \ref{thm:b-dense}  uses a new result of Gao, the third author, and Yip to show that if the set of integer parts of a sequence has positive upper Banach density, then the corresponding set $\mathbf{E}$ is countable and therefore $\mathbf{EE}^{-1} \neq \R$. 

Lemma \ref{lem:construction-lemma-gmy24} also yields a one-line proof of Theorem \ref{th:bkm}: if $a_n \to \infty$ is linear,  then $\mathbf{EE}^{-1} = \Q \neq \R$. 


For exponential sequences (e.g. $(2^n)_{n=1}^\infty$) we do not know if $\mathbf{EE}^{-1} \neq \R$, so Lemma \ref{lem:construction-lemma-gmy24} may not be applicable. However, in the range $0 \leq p < 1/2$, we can modify the ideas discussed above to construct a $p$-large set containing no affine copy of $(2^n)_{n=1}^\infty$. Our substitute for estimates on the size of $\mathbf{E}$ is the fact that for every $x \in \R$ and $y \in \R \setminus \Q(2)$, the sequence $(x + y2^n)_{n=1}^\infty$ modulo $1$ is not contained in any interval of length smaller than $1/2$. This observation is due to A. Dubickas (specifically, \cite[Theorem 1]{Dubickas06} and \cite[Theorem 2]{D06b}). Combining it with Theorem \ref{lem:construction2} (with $C = \Q(2)$) proves Corollary \ref{cor:2nlarge}.

\section{The relationship between the original Erd\H{o}s similarity conjecture and its variant ``in the large''.}\label{sec:connections} It is currently not known if the non-universality of sequences in $\mathscr{R}(p)$ implies the non-universality of sequences in sets of positive measure, or vice versa. However, in \cite{BGKMW23}, A. Burgin, S. Goldberg, T. Keleti, C. MacMahon, and X. Wang showed that some of the results of subsection \ref{sec:results-large} can be used to construct sets of positive measure not containing any decreasing geometric progressions. 

\begin{theorem}[Burgin--Goldberg--Keleti--MacMahon--Wang]\label{th:bgkmw-erdos}
    There exists a measurable set $E \subset [0,1]$ of positive measure, with $0$ as a Lebesgue density point; i.e. 
    \begin{align*}
        \lim_{r \to 0^+} \frac{m\big(E \cap (0,r)\big)}{r} = 1,
    \end{align*}
    and containing no sequence $(yb^{-n})_{n = 1}^\infty$ for any $x, b \in \R$ with $y \neq 0, b > 1$.
\end{theorem}

To put this result in context, recall that the following special case of the Erd\H{o}s similarity conjecture is open.

\begin{question}
    Let $b > 1$ be a real number. Is there a set of positive measure containing no sequence $(x + yb^{-n})_{n=1}^\infty$ for any $x, y \in \R$ with $y \neq 0$?
\end{question}

Theorem \ref{th:bgkmw-erdos} provides a single set that gives an affirmative answer for $x = 0$ but all $b > 1$. If one wanted to prove that a set of positive measure \textit{does} contain some given rapidly-decreasing sequence, a reasonable strategy would be to look for such a sequence near a Lebesgue density point (see the discussion in Section \ref{section 2} for bi-Lipschitz copies or affine copies of finite sets).   Theorem \ref{th:bgkmw-erdos} shows that this strategy fails for exponential sequences. 

The key step in the proof of Theorem \ref{th:bgkmw-erdos} is the following lemma.

\begin{lemma}[Burgin--Goldberg--Keleti--MacMahon--Wang]\label{lem:bgkmw}
    Let $a_n \to \infty$ be an increasing sequence of real numbers. If for each $0 \leq p < 1$, there exists a $p$-large set containing no affine copy of $(a_n)_{n=1}^\infty$, then there exists a set $F \subset \R$ such that
        \begin{align}\label{eq:keleti-size}
        \lim_{k \to \infty} m(F \cap [k, k+1]) = 1
        \end{align}
    and containing no affine copy of $(a_n)_{n=1}^\infty$.
\end{lemma}

\begin{proof}[Proof of Theorem \ref{th:bgkmw-erdos}]
    Theorem \ref{th:bkm}, together with the above lemma, allows us to find a set $F \subset \R$ such that \eqref{eq:keleti-size} holds, and that contains no sequence $(x + yn)_{n=1}^\infty$ for any $x, y \in \R$ with $y \neq 0$. Clearly the set $E = \exp(-F)$ cannot contain any sequence $(cb^{-n})_{n = 1}^\infty$ for all $c, b \in \R$ (where $c = e^{-x}$, $b = e^{y}$). Finally, using \eqref{eq:keleti-size}, it is straightforward to show that $E$ has $0$ as a Lebesgue density point.
\end{proof}

Lemma \ref{lem:bgkmw} connects the results on Question \ref{q:large-erdos} to the Székely-type problems mentioned at the start of Section \ref{sec:large}. Indeed, if a measurable set $E \subseteq \R$ satisfies \eqref{eq:keleti-size}, then it is $p$-large for \textit{every} $0 \leq p < 1$. It follows that $E$ has upper density $1$. By combining Lemma \ref{lem:bgkmw} with the theorems in Subsection \ref{sec:results-large}, we obtain many examples of \textit{infinite} configurations that are not universal among sets with full upper density. In particular, this shows that upper density is not sufficient to guarantee the existence of ``all large copies'' of these infinite configurations.

\noindent{\bf Acknowledgments.} Chun-Kit Lai would like to thank his former graduate students, Angel Cruz, John Gallagher, Lekha Priya Patil, Danny Lopez, Nicholas Mendler and Yeonwook Jung (one of the authors in this article), for all the wonderful discussions about the Erd\H{o}s similarity problem over the last few years at SFSU. The authors would also like to thank professors De-Jun Feng,  Mihalis Kolountzakis, Malabika Pramanik and Ying Xiong for their support and discussions. 

\noindent{\bf Funding.} Chun-Kit Lai is partially supported by the AMS-Simons Research Enhancement Grants for Primarily Undergraduate Institution (PUI) Faculty.

\bibliographystyle{amsalpha}
\bibliography{biblio.bib}

\author{Yeonwook Jung}
\address{University of California, Irvine, Department of Mathematics, Rowland Hall, Irvine, CA 92697}
\curraddr{}
\email{yeonwoj1@uci.edu}
\thanks{}

\author{Chun-Kit Lai}
\address{San Francisco State University Department of Mathematics, 1600 Holloway Ave, San Francisco, CA 94132
}
\curraddr{}
\email{cklai@sfsu.edu}
\thanks{}

\author{Yuveshen Mooroogen}
\address{University of British Columbia, 2329 West Mall
Vancouver, BC Canada V6T 1Z4}
\email{yuveshenm@math.ubc.ca}





\end{document}